\documentclass[12pt]{article}

\usepackage{amsmath}
\usepackage{amssymb}
\usepackage{amscd}
\usepackage{amsthm}
\theoremstyle{plain}
\newtheorem{theorem}{Theorem}
\newtheorem{lemma}{Lemma}
\newtheorem{prop}{Proposition}
\newtheorem{case}{Case}
\newtheorem{subcase}{Subcase}
\theoremstyle{remark}
\newtheorem{Rem}{Remark}
\def\Hilb{\mathop{\mathrm{Hilb}}}
\def\codim{\mathop{\mathrm{codim}}}
\def\Supp{\mathop{\mathrm{Supp}}}

\begin{document}
\title{A note on Tsuji's criterion \\ for numerical triviality}
\author{ Shigetaka FUKUDA \\
Faculty of Education, Gifu Shotoku Gakuen University}

\date{\empty}

\maketitle \thispagestyle{empty}
\pagestyle{myheadings}
\markboth{Shigetaka FUKUDA}{Numerical triviality}
\begin{abstract}
In this study, we give an alternative and elementary proof to Tsuji's criterion for a Cartier divisor to be numerically trivial.
\end{abstract}

\section{Introduction}

In this article every algebraic variety is proper over the field of complex numbers $\mathbf{C}$.

In 1970's Iitaka (\cite{Ii}) initiated the classification theory of higer dimensional algebraic varieties by using the pluricanonical systems.
In 1980's Mori (\cite{Mo}) deepened the Iitaka theory by cutting off the subvarieties of elliptic type. 

In \cite{Tsuji}, Tsuji gave an interesting and useful criterion for a Cartier divisor to be numerically trivial:

\begin{theorem}[Tsuji {\cite[Lemma 5.1]{Tsuji}}, cf.\ Bauer et al.\ {\cite[Theorem 2.4]{BCEKPRSW}}]\label{thm:Tsuji}
Let $f:M \to B$ a surjective morphism between complete varieties.
Let $L$ be a nef Cartier divisor on $M$, $W$ some subvariety of $M$ such that $f(W)=B$ and $B_0$ a subset of $B$ which is a union of countably many proper Zariski-closed subsets.
Assume that:
 \begin{description}
  \item[(1)] for some $b \in B$, $(L,C) = 0$ for every curve $C$ on $f^{-1} (b);$
  \item[(2)] $(L,C) = 0$ for every irreducible curve $C$ on $W$ such that $f(C) \nsubseteq B_0.$
 \end{description}
Then $L$ is numerically trivial.
\end{theorem}

Tsuji's criterion for numerical triviality is one of the basic tools to decompose every algebraic variety into the varieties of elliptic type, of parabolic type and of hyperbolic type by cutting off the varieties of parabolic type.
(See Ambro (\cite{Am04}, \cite{Am05}) and Fukuda (\cite{Fk}).)

Tsuji's proof (\cite{Tsuji}) is analytic and the proof (\cite{BCEKPRSW}) by Th.\ Bauer, F. Campana, Th.\ Eckl, S. Kebekus, Th.\ Peternell, S. Rams, T. Szemberg and L. Wotzlaw is algebraic.

In this research note we give an alternative and elementary proof to Tsuji's criterion (Theorem \ref{thm:Tsuji}).
The argument (see Subcase \ref{subsubcas:genl}), which uses the following corollary of the Hodge index theorem, due to Bauer et al.\ is essential:

\begin{lemma}[cf.\ {\cite[Proposition 2.5]{BCEKPRSW}}]\label{lem:BCEKPRSW}
Let $f:M \to B$ a surjective morphism from a complete surface $M$ to a complete curve $B$.
Let $L$ be a nef Cartier divisor on $M$.
Assume that:
 \begin{description}
  \item[(1)]    for some $b \in B$, $(L,C) = 0$ for every curve $C$ on $f^{-1} (b);$
  \item[(2)]    $(L,W) = 0$ for some irreducible curve $W$ on $M$ such that $f(W) = B$.
 \end{description}
Then $L$ is numerically trivial.
\end{lemma}

\begin{Rem}
In the statement of Lemma \ref{lem:BCEKPRSW}, the condition (1) immediately implies that $L$ is numerically trivial on every general fiber of the morphism $f$, by considering the flattening.
By the normalization, the Stein factorization and the desingularization, the article (\cite[Proposition 2.5]{BCEKPRSW}), for an algebraically fibered surface, implies the assertion of Lemma \ref{lem:BCEKPRSW}.
\end{Rem}

\section{Elementary proof of Main Theorem \ref{thm:Tsuji}}

\begin{proof}
We prove the assertion by induction on $(\dim M, \dim B)$.

First we take a commutative diagram Figure \ref{fig:CDStein}:

\clearpage 
\begin{figure}[h]
$$\begin{CD}
M' @>\text{$g$}>> B' \\
@V\text{$\mu$}VV @VV\text{$\nu$}V \\
M @>\text{$f$}>> B \\
\end{CD}$$
\caption{The Stein factorization for $f$}
\label{fig:CDStein}
\end{figure}

with the following properties:
 \begin{description}
  \item[(1)] $M'$ and $B'$ are nonsingular projective varieties;
  \item[(2)] $\mu$ is a birational morphism;
  \item[(3)] $\nu$ is a generically finite morphism;
  \item[(4)] $g$ is a morphism with only connected fibers.
 \end{description}
There exists some irreducible component $W'$ of $\mu^{-1} (W)$ such that $g(W') = B'$.
We set $L' := \mu^{*} L$.

The locus $g(\bigcup \{ C' \vert C'$ is an irreducible curve on $M'$, $g(C')$ is a point and the intersection number $(L', C')>0 \} )$ is included in a union of at most countably many proper Zariski-closed subsets of $B'$ (see Proposition \ref{prop:genl fib}).
Thus we obtain a union $B_0' \supseteq \nu^{-1} (B_0)$ of countably many proper Zariski-closed subsets of $B'$ with the following two properties:
 \begin{description}
  \item[(1)] $L'$ is numerically trivial on every fiber of $g$ over $B' \setminus B_0'$;
  \item[(2)] $(L', C'_W) = 0$ for every irreducible curve $C'_W$ on $W'$ such that $g(C'_W) \nsubseteq B_0'$.
 \end{description}

It suffices to prove that $(L',C') = 0$ for every irreducible curve $C'$ on $M'$.
We fix an irreducible curve $C'$ on $M'$

\begin{case}\label{cas:nsub}
$g(C') \nsubseteq B_0'$.
This case divides into Subcases \ref{subcase 1} and \ref{subcase 2}.
\end{case}

\begin{subcase}\label{subcase 1}
$g(C') \nsubseteq B_0'$ and $g(C')$ is a point.
\end{subcase}
We have $(L',C') = 0$ from (1).

\begin{subcase}\label{subcase 2}
$g(C') \nsubseteq B_0'$ and $g(C')$ is a curve.
This subcase divides into Subcases \ref{subcase 3} and \ref{subsubcas:genl}.
\end{subcase}

\begin{subcase}\label{subcase 3}
$g(C') \nsubseteq B_0'$, $g(C')$ is a curve and $\dim B = 1$.
This subcase divides into Subcases \ref{subcase 4}, \ref{subcase 5} and \ref{subcase 6}.
\end{subcase}

We note that $g(W_1)=g(C')=B'$ for some irreducible curve $W_1$ on $W'$.

\begin{subcase}\label{subcase 4}
$g(C') \nsubseteq B_0'$, $g(C')$ is a curve, $\dim B = 1$ and $\dim M =1$.
\end{subcase}

$M' = W' = W_1 = C'$.
Thus $(L', C') = 0$.

\begin{subcase}\label{subcase 5}
$g(C') \nsubseteq B_0'$, $g(C')$ is a curve, $\dim B = 1$ and $\dim M = 2$.
\end{subcase}

Because $(L', W_1) = 0$ from (2), Lemma \ref{lem:BCEKPRSW} implies that $L'$ is numerically trivial and thus $(L',C') = 0$.

\begin{subcase}\label{subcase 6}
$g(C') \nsubseteq B_0'$, $g(C')$ is a curve, $\dim B = 1$ and $\dim M \geqq 3$.
\end{subcase}

Because the codimension $\codim (W_1 \cup C', M') \geqq 2$, we have an irreducible hyperplane section $H$ of $M'$ that includes $W_1$ and $C'$ (see Proposition \ref{prop:hyp plane}).
Then $L' \mid_H$ is numerically trivial from the induction hypothesis.
Consequently $(L',C') = 0$.

\begin{subcase}[cf.\ {\cite[2.1.2]{BCEKPRSW}}]\label{subsubcas:genl}
$g(C') \nsubseteq B_0'$, $g(C')$ is a curve and $\dim B \geqq 2$.
\end{subcase}

Let $S:=\{ S_i \}$ be the set of irreducible components of $g^{-1}(g(C'))$.
We note that $g(W_1)=g(C')$ for some irreducible curve $W_1$ on $W'$.
Thus $L' \mid_{S_1}$ is numerically trivial for some $S_1 \in S$ such that $S_1 \supseteq W_1$ from the property (2) and from the induction hypothesis.

If $g(\cup_{m \ne 1} S_m) = g(C')$, then $g(S_1 \cap (\cup_{m \ne 1} S_m)) = g(C')$ from the connectedness of fibers of $g$ and therefore $g(S_1 \cap S_2) = g(C')$ for some $S_2 \in S$.

Thus $\dim g(\cup_{m \ne 1} S_m) \leqq 0$ or $g(S_1 \cap S_2) = g(C')$.

If $g(S_1 \cap S_2) = g(C')$ and $g(\cup_{m \ne 1, 2} S_m) = g(C')$, then $g((S_1 \cup S_2) \cap (\cup_{m \ne 1, 2} S_m)) = g(C')$ from the connectedness of fibers of $g$ and therefore $g((S_1 \cup S_2) \cap S_3) = g(C')$ for some $S_3 \in S$.
From this argument, we obtain the following properties:
 \begin{description}
  \item[(1)] $g(S_1) = g(S_2) = \cdots = g(S_k) = g(C')$;
  \item[(2)] $g((S_1 \cup S_2 \cup \cdots \cup S_{i-1}) \cap S_i) = g(C')$ for all $i$ with $1 \leqq i \leqq k$;
  \item[(3)] $\dim g(\cup_{i \ne 1, 2, \cdots k} S_i) \leqq 0.$
 \end{description}

The fact that $L' \mid_{S_1}$ is numerically trivial and that $g(S_1 \cap S_2) = g(C')$ implies that $L' \mid_{S_2}$ is numerically trivial from the induction hypothesis.
The fact that $L' \mid_{S_1 \cup S_2}$ is numerically trivial and that $g((S_1 \cup S_2) \cap S_3) = g(C')$ implies that $L' \mid_{S_3}$ is numerically trivial from the induction hypothesis.
This argument implies that $L' \mid_{S_1 \cup S_2 \cup \cdots \cup S_k}$ is numerically trivial.
Because $\dim g(\cup_{i \ne 1, 2, \cdots k} S_i) \leqq 0$, we have that $C' \subseteq S_1 \cup S_2 \cup \cdots \cup S_k$.
Consequently $(L',C') = 0$.

\begin{case}
$g(C') \subseteq B_0'$.
This case divides into Subcases \ref{subcase 8} and \ref{subcase 11}.
\end{case}

\begin{subcase}\label{subcase 8}
$g(C') \subseteq B_0'$ and $\dim M = 2$.
This subcase divides into Subcases \ref{subcase 9} and \ref{subcase 10}.
\end{subcase}

\begin{subcase}\label{subcase 9}
$g(C') \subseteq B_0'$, $\dim M =2$ and $\dim B = 1$.
\end{subcase}

Lemma \ref{lem:BCEKPRSW} implies that $L'$ is numerically trivial and thus $(L', C') = 0$.

\begin{subcase}\label{subcase 10}
$g(C') \subseteq B_0'$, $\dim M =2$ and $\dim B = 2$.
\end{subcase}

Because $W' = M'$, we have that $(L',H) = 0$ for an irreducible hyperplane section $H$ of $M'$ from the property (2) of the divisor $L'$.
The Hodge index theorem implies that $L'$ is numerically trivial.
Thus $(L', C') = 0$.

\begin{subcase}\label{subcase 11}
$g(C') \subseteq B_0'$ and $\dim M \geqq 3$.
\end{subcase}

Because the codimension $\codim (C', M') \geqq 2$, there exists an irreducible hyperplane section $H$ of $M'$ that includes $C'$ (see Proposition \ref{prop:hyp plane}).
We may assume that $g(H) \nsubseteq B_0'$.
Note that, from Case \ref{cas:nsub}, $(L,C'')=0$ for every irreducible curve $C''$ on $H$ such that $g(C'') \nsubseteq B_0'$.
Thus $L' \mid _H$ is numerically trivial from the induction hypothesis.
Consequently $(L',C')=0$.

\end{proof}

\section{Appendix}

In this appendix, we state two elementary propositions and their proofs, which are  well known to the experts, for the readers' convinience.

\begin{prop}\label{prop:genl fib}
Let $f \colon M \to B$ be a surjective morphism between projective varieties and $L$ a nef Cartier divisor on $M$.
We assume that, for some $b \in B$, the intersection number $(L,C) = 0$ for every irreducible curve $C$ on $f^{-1} (b)$.

Then the locus $f(\bigcup \{ C \vert C$ is an irreducible curve on $M$, $f(C)$ is a point and the intersection number $(L,C)>0 \} )$ is included in a union of at most countably many proper Zariski-closed subsets of $B$.
\end{prop}

\begin{proof}

There exists some ample divisor $A$ on $B$.
Assume that $C$ is an irreducible curve on $M$ such that $(f^* A, C) = 0$ (i.e.\ $f(C)$ is a point) and that $(L, C) >0$.
There exists some irreducible component $W$ of the universal scheme for the Hilbert scheme $\Hilb (M)$ of $M$ such that $W$ includes $p_2^{-1} ([C])$ where $[C]$ is the point ($\in \Hilb (M)$) which represents the subscheme $C$ of $M$ and such that the diagram Figure \ref{fig:CDdeform} holds:
\clearpage 
\begin{figure}[h]
$$\begin{CD}
W @>\text{embedding}>> M \times \Hilb (M) @>\text{$p_2$}>> \Hilb (M)    \\
@. @VV\text{$p_1$}V    \\
@. M
\end{CD}$$
\caption{The deformation of $C$ in $M$}
\label{fig:CDdeform}
\end{figure}
with the projections $p_1$ and $p_2$ and with the property that  $\dim W = \dim p_2 (W) + 1$.
We set $T: = p_2 (W)$.

First we consider the normalization $n_1 \colon W_n \to W$, $n_2 \colon T_n \to T$ and $n_3 \colon W_n \to T_n$  of the morphism $p_2 \vert_W \colon W \to T$.

Next consider the Stein factorization $W_n \stackrel{s_1}{\to} T' \stackrel{s_2}{\to} T_n$ of the morphism $n_3 \colon W_n \to T_n$.

Lastly consider the flattening $f_1 \colon W'' \to W_n$, $f_2 \colon T'' \to T'$ and $f_3 \colon W'' \to T''$ of the morphism $s_1 \colon W_n \to T'$, where the morphism $f_2$ is birational and the variety $T''$ is nonsingular.
We note that the morphism $f_3 \colon W'' \to T''$ is flat and with only connected fibers.

We put $h \colon = n_1 f_1$.

Thus we have the commutative diagram Figure \ref{fig:CDflattening}:

\clearpage 
\begin{figure}[h]
$$\begin{CD}
W'' @>\text{$f_3$}>> T'' \\
@V\text{$f_1$}VV @VV\text{$f_2$}V \\
W_n @>\text{$s_1$}>> T' \\
@| @VV\text{$s_2$}V \\
W_n @>\text{$n_3$}>> T_n \\
@V\text{$n_1$}VV @VV\text{$n_2$}V \\
(p_2^{-1} ([C]) \subset) W \qquad \qquad @>\text{$p_2 |_{W}$}>> \quad \qquad T (\owns [C]) \\
@V\text{embedding}VV @VV\text{embedding}V \\
M \times \Hilb (M) @>\text{$p_2$}>> \Hilb (M) \\
@VV\text{$p_1$}V \\
M \\
@VV\text{$f$}V \\
B
\end{CD}$$
\caption{The flattening of the deformation}
\label{fig:CDflattening}
\end{figure}

From the flatness of the morphism $f_3 \colon W'' \to T''$, the intersection number $(h^* p_1^* f^* A, F'') = 0$ for every fiber $F''$ of the morphism $f_3 \colon W'' \to T''$, because $(p_1^* f^* A, p_2^{-1} ([C])) = 0$.
Thus, for every fiber $F''$ of the morphism $f_3 \colon W'' \to T''$, the morphism $f$ contracts $p_1 h (F'')$ to one point from the connectedness of $F''$.
In other words, $p_1 h (F'')$ is included in some fiber of $f$.

There exists some ample divisor $A'$ on $W$.
Of course $(A', F) > 0$ for every curve $F$ on $W$.
Because the morphism $h$ is birational, we have that $h(F''')$ is not a point (i.e.\  $(h^* A', F''') > 0$) for some fiber $F'''$ of $f_3 \colon W'' \to T''$.
From the flatness, $(h^* A', F'')>0$ for every fiber $F''$ of $f_3 \colon W'' \to T''$.
Thus every fiber $F''$ of $f_3 \colon W'' \to T''$ cannot be contracted to a point by the morphism $h$.

There exists some fiber $F''_1$ of $f_3 \colon W'' \to T''$ such that $h^{-1}(p_2^{-1} ([C])) \cap F''_1 \ne \emptyset$.
Then $F''_1 \subset h^{-1}(p_2^{-1} ([C])$, because the morphism $p_2 h$ maps $F''_1$ to a point $[C] \in T$.
Consequently $h(F''_1) = p_2^{-1} ([C])$, because $F''_1$ does not contract to a point by the morphism $h$.
Thus $(h^* p_1^{*} L, F''_1) > 0$.
From the flatness of the morphism $f_3 \colon W'' \to T''$, the intersection number $(h^* p_1^{*} L, F'') > 0$ for every fiber $F''$ of the morphism $f_3 \colon W'' \to T''$.

We note that every fiber of $f_3 \colon W'' \to T''$ is mapped in some fiber of $p_2 \vert_W \colon W \to T$.
In other words, every fiber of $p_2 \vert_W$ is swept out by fibers of $f_3$.

So, for every fiber $F$ of $p_2 |_W \colon W \to T$, the locus $p_1 (F)$ is swept out by connected curves $C'$ such that $f(C')$ is one point and that the intersection number $(L, C') > 0$.
(We note that we consider $p_1 h(F'')$ as $C'$ and that $C' = p_1 h (F'') = \Supp ((p_1 h)_* F'')$ from the connectedness of $F''$.)
Thus $b \notin f(p_1 (F))$.
Consequently $b \notin f(p_1 (W))$.
In other words, $p_1(W)$ is disjoint with $f^{-1}(b)$.

The countability of the irreducible components of the Hilbert scheme $\Hilb (M)$ of $M$ implies the assertion.

\end{proof}

\begin{prop}\label{prop:hyp plane}
Let $M$ be a nonsingular projective variety and $C$ a Zariski-closed subset with codimension $\codim (C, M) \geq 2$.
Then there exists some irreducible hyperplane section $H$ such that $H \supset C$.
\end{prop}

\begin{proof}
We take some ample divisor $A$ on $M$.
We have a birational morphism $f \colon M' \to M$ such that $M'$ is a nonsingular projective variety, that $f^{-1} (C)$ is divisorial with only simple normal crossings and  that there exists an effective divisor $C_0$ with the property that $\Supp (C_0) = f^{-1} (C)$ and $- C_0$ is $f$-ample.
Then $mf^* A - C_0$ is ample for a sufficiently large integer $m$.
For a sufficiently large and divisible integer $l$, the divisor $lmA$ is very ample and there exists a member $H_0 \in \vert l(mf^* A - C_0) \vert$ which is very ample and irreducible.
We put $H \colon = f_* (H_0 + l C_0)$.
Then $H \in \vert l m A \vert$.

The locus $f^{-1}(H)$ coincides with $\Supp (f^* H) = \Supp(H_0 + l C_0)$.
Thus $f^{-1}(H) \supset f^{-1} (C)$.
\end{proof}

\noindent{\bf Data Availability}

\noindent
No data were used to support this study.

\noindent{\bf Conflict of Interests}

\noindent
The author declares that there are no conflicts of interest.

\noindent{\bf Acknowledgments}

\noindent
In this series of research, the author was supported by the research grant of Gifu Shotoku Gakuen University in the years 2019 and 2020.
The author would like to thank the referee who carefully read the paper and gave the suggestions to improve the explanation.

\bigskip
Faculty of Education, Gifu Shotoku Gakuen University

Yanaizu-cho-Takakuwa-Nishi, Gifu, Gifu Prefecture 501-6194, Japan

fukuda@gifu.shotoku.ac.jp

\end{document}